\documentclass[11pt]{amsart}
\usepackage{amsmath,amsthm,amsfonts,amssymb,verbatim, eucal}

\numberwithin{equation}{section} 
\newtheorem{thm}[equation]{Theorem}
\newtheorem{prop}[equation]{Proposition}
\newtheorem{lemma}[equation]{Lemma}
\newtheorem{cor}[equation]{Corollary}
\newtheorem{example}[equation]{Example}
\newtheorem{remark}[equation]{Remark}
\newtheorem{definition}[equation]{Definition}
\newtheorem{defns}[equation]{Definitions}

\newenvironment{rem}{\begin{remark}\rm}{\end{remark}}

\DeclareMathOperator{\coop}{coop}
\DeclareMathOperator{\op}{op}

\newcommand{\DOT}{\setlength{\unitlength}{1pt}\begin{picture}(2.5,2)
               (1,1)\put(2,3.5){\circle*{3}}\end{picture}}
\newcommand{\e}{\epsilon}
\newcommand{\db}{\DOT_{\beta}}
\renewcommand{\_}[1]{\mbox{$_{\left( #1 \right)}$}}
\renewcommand{\k}{{\mathbb K}}

\newcommand{\coh}{{\rm H}}
\newcommand{\Z}{{\mathbb Z}}
\newcommand{\N}{{\mathbb N}}
\newcommand{\id}{\mbox{\rm id\,}}
\newcommand{\Hom}{\mbox{\rm Hom\,}}
\newcommand{\half}{\frac{1}{2}}
\newcommand{\w}{\omega}
\newcommand{\ufn}{{\mathfrak{u}}_q({\mathfrak{sl}}_n)}
\newcommand{\ursn}{{\mathfrak{u}}_{r,s}({\mathfrak{sl}}_n)}
\newcommand{\urs}{{\mathfrak{u}}_{r,s}({\mathfrak{sl}}_3)}

\newcommand{\E}{\mathcal E}
\newcommand{\F}{\mathcal F}

\newcommand{\YD}{ ${\,}_H\mathcal{YD}^H$}
\newcommand{\Id}{\mathop{\rm Id}}
\renewcommand{\Im}{\mathop{\rm Im\,}}
\renewcommand{\b}{\mathfrak{b}}

\newcommand{\ot}{\otimes}

\begin{document}

\title{Yetter-Drinfeld Modules Under Cocycle Twists}

\author{Georgia Benkart}
\address{Department of Mathematics, University of Wisconsin,
Madison, Wisconsin 53706, USA}
\email{benkart@math.wisc.edu}

\author{Mariana Pereira}
\address{Centro de Matem\'atica, Facultad de Ciencias,
Igu\'a 4225, Montevideo 11400, Uruguay}
\email{mariana@cmat.edu.uy}

\author{Sarah Witherspoon}
\address{Department of Mathematics, Texas A\&M University,
College Station, Texas 77843, USA}
\email{sjw@math.tamu.edu}

\date{October 14, 2009}

\thanks{ 
{\em To appear in Journal of Algebra}}

\maketitle
\begin{center} {\emph{To Susan Montgomery in honor of her distinguished career}}
\end{center}
\begin{abstract}
We give an explicit formula for the correspondence between simple Yetter-Drinfeld
modules for certain finite-dimensional
pointed Hopf algebras $H$ and those for cocycle twists
$H^{\sigma}$ of $H$. This implies an equivalence between modules
for their Drinfeld doubles.   To illustrate our results, we consider the restricted
two-parameter quantum groups ${\mathfrak{u}}_{r,s}({\mathfrak{sl}}_n)$
under conditions on the parameters  guaranteeing that
${\mathfrak{u}}_{r,s}({\mathfrak{sl}}_n)$ is a Drinfeld double of its
Borel subalgebra. We  determine
explicit correspondences between $\ursn$-modules for different values of
$r$ and $s$  and provide examples where no such correspondence can exist.
Our examples were obtained via the computer algebra system {\sc Singular::Plural}.
\end{abstract}

\section{Introduction}
Radford \cite{radford} gave a construction of simple Yetter-Drinfeld modules for
a pointed Hopf algebra $H$,  whose group $G(H)$ of grouplike elements is abelian,
under fairly general hypotheses on $H$.
These simple modules are in one-to-one correspondence with the Cartesian product of
$G(H)$ with its dual group, and are realized as vector subspaces of the Hopf algebra $H$  itself.
Radford and Schneider \cite{radford-schneider} generalized this method to include
a much wider class of Hopf algebras given by cocycle twists of
tensor product Hopf algebras.
Again the simple modules are in one-to-one correspondence with a set of
characters, but this time each simple module is realized as the quotient
of a Verma module by its unique maximal submodule
(equivalently, by the radical of the Shapovalov form)
reminiscent of standard Lie-theoretic methods.
Thus,  the work of Radford and Schneider ties Radford's explicit
realization of simple modules as vector subspaces of the Hopf
algebra to more traditional methods.
One advantage of Radford's approach is that it is purely Hopf-theoretic,
and so a priori there are no restrictions on parameters as often
occur in Lie theory.

In this paper, we determine how Radford's construction behaves under cocycle twists.
We give a precise correspondence between simple Yetter-Drinfeld modules of $H$
and those of  the cocycle twist $H^\sigma$   for a large class
of Hopf algebras $H$ (see Theorem \ref{modules-twist} below).
This result uses a theorem of
Majid and Oeckl \cite{majid-oeckl} giving a category equivalence between
Yetter-Drinfeld modules for $H$ and those for $H^{\sigma}$.
We then apply a result of
Majid \cite{majid} relating Yetter-Drinfeld $H$-modules
for a finite-dimensional Hopf algebra $H$  to modules
for the Drinfeld double $D(H)$ of $H$.
In this way we obtain an explicit formulation of
a category equivalence between $D(H)$-modules
and $D(H^{\sigma})$-modules.

In Sections 4 and 5,
we specialize to the setting of restricted quantum groups.
We use  Radford's construction to analyze in detail
the simple modules for the restricted two-parameter quantum groups $\mathfrak{u}_{r,s}(\mathfrak{sl}_n)$;
that is,  two-parameter versions of the finite-dimensional quantum group
${\mathfrak{u}}_q({\mathfrak{sl}_n})$, where  $q$ is a root of unity
(in fact, ${\mathfrak{u}}_q({\mathfrak{sl}_n})$ is a quotient of
$\mathfrak{u}_{q,q^{-1}}(\mathfrak{sl}_n)$).
Some conditions on the parameters $r$ and $s$ are known under which
a finite-dimensional two-parameter quantum group
$\ursn$ is the Drinfeld double of a Borel subalgebra
(see e.g.\ \cite[Thm.~4.8]{benkart-witherspoon} for the case ${\mathfrak{g}}=
{\mathfrak{sl}}_n$, \cite[Thm.\ 6.2]{hu-wang} for the case ${\mathfrak{g}}
= {\mathfrak{so}}_{2n+1}$,
 and  \cite[Ex.~3.13]{burciu} for the case $s=r^{-1}$).
When the conditions hold, ${\mathfrak{u}}_{r,s}({\mathfrak{g}})$-modules
correspond to Yetter-Drinfeld
modules for the Borel subalgebra, and these in turn are given by Radford's construction.
For $n \geq 3$ (under a mild assumption)
there is no Hopf algebra isomorphism between $\ursn$ and
 ${\mathfrak{u}}_{q,q^{-1}} (\mathfrak{sl}_n)$  unless $r= q^{\pm 1}, s = q^{\mp 1}$,
(see \cite[Thm.~5.5]{hu-wang})
On the other hand,  computations in {\sc Singular::Plural}
 show that the representations of $\ursn$ and $\mathfrak{u}_{q,q^{-1}}(
\mathfrak{sl}_n)$ can be quite
similar when the parameters are related in certain ways.
We give a precise explanation for this similarity
using our results on cocycle twists.
In Theorem \ref{prop:findim},  we exploit  an explicit cocycle twist that
yields an equivalence of categories of Yetter-Drinfeld modules
for the respective Borel subalgebras $H_{r,s}$ and
$H_{q,q^{-1}}$, and hence an equivalence of the categories of modules
for $\ursn$ and $\mathfrak{u}_{q,q^{-1}}(\mathfrak{sl}_n)$, under some
conditions on the parameters.
Cocycle twists different from the ones used here have been shown
to give rise to other multi-parameter versions of
quantum groups  in earlier
work by Reshetikhin \cite{reshetikhin}, by Doi \cite{doi93},
and by Chin and Musson \cite{chin-musson}.

For particular choices of the parameters, however, there is no such
cocycle twist, and in that situation the representation theories of $\ursn$ and of  $\mathfrak{u}_{q,q^{-1}}(
\mathfrak{sl}_n)$ can be quite different.
As an example of this phenomenon, we have used the
computer algebra system {\sc Singular::Plural}
to show in Example \ref{example1} below that for $q$ a
primitive fourth root of unity, the dimensions of the simple
modules for ${\mathfrak{u}}_{1,q}({\mathfrak{sl}}_3)$ and for
${\mathfrak{u}}_{q,q^{-1}}({\mathfrak{sl}}_3)$ differ significantly.
For a wide class of such examples, Radford's construction lends itself
to computations using Gr\"obner basis techniques.   We further illustrate
the construction by briefly explaining how to use {\sc Singular::Plural} to compute bases and dimensions
of {\em all} simple modules for some of the one- and two-parameter
quantum groups  ${\mathfrak{u}}_{q}({\mathfrak{sl}}_3), \urs$.
Many such computations appeared in the second author's Ph.D.\
thesis \cite{pereira1}.

\section{Preliminaries}

\noindent {\bf Drinfeld doubles and Yetter-Drinfeld modules}
\medskip

Let $H$ be a finite-dimensional Hopf algebra over a field $\k$ with
coproduct $\Delta$, counit $\varepsilon$, and  antipode $S$.
In this paper, we generally assume that $\k$ is algebraically closed of
characteristic 0, although this is not needed for the definitions.

The {\em Drinfeld double of $H$}, denoted $D(H)$, is the Hopf algebra defined to be
$$D(H)=(H^*)^{\coop}\otimes H$$ as a coalgebra, and the algebra structure is given by
$$ (f\otimes a)(f' \otimes b) = \sum f\left( a\_1 \rightharpoonup f' \leftharpoonup S^{-1}(a\_3) \right) \otimes a\_2b,$$
for all $f,f'\in H^*$ and $a,b \in H$, where
$\langle x \rightharpoonup f \mid  y\rangle = \langle f\mid yx \rangle $
and $\langle f \leftharpoonup x \mid y \rangle =\langle f \mid xy\rangle$,
for all $x,y \in H$ and $f \in H^*$.   In these expressions  $\langle\,\, \mid \,\,\rangle$ is the
natural pairing between $H$ and the dual Hopf algebra $H^* = \hbox{\rm Hom}_\k(H,\k)$, and $(H^*)^{\coop}$  is $H^*$  with the opposite coproduct.
(In the above equation and throughout the paper, we are adopting Sweedler notation.)
With this structure and a suitable counit and antipode, $D(H)$ is a
Hopf algebra. (See \cite[Defn.\ 10.3.5]{Mo} for details.)

For any bialgebra $H$, a {\em left-right Yetter-Drinfeld $H$-module} $(M, \cdot, \delta)$
is a left $H$-module $M$ (with action denoted $a \cdot m$ for $a\in H$ and $m\in M$)
that is also a right $H$-comodule via $\delta:M
\rightarrow M\ot H$, written $\delta(m)=\sum m_{(0)}\ot m_{(1)}$ for
all $m\in M$, such that
the following compatibility condition is satisfied:
$$\sum a\_1\cdot m\_0 \otimes a\_2 m\_1 = \sum ( a\_2\cdot m)\_0 \otimes (a\_2\cdot m)\_1 a\_1$$
for all $a\in H$, $m\in M$.
There is also a notion of a right-left Yetter-Drinfeld module, which we do not consider here.

{\it In what follows, all references to Yetter-Drinfeld modules are to left-right ones and
all modules considered are left modules.}

The category of Yetter-Drinfeld modules over a bialgebra $H$ will be denoted by \YD
and the isomorphism class of  $M\in$\YD  \ by $[M]$.
Yetter-Drinfeld $H$-modules are  $D(H)$-modules
and conversely as a result of the following theorem, which is \cite[Prop.\ 2.2]{majid}
(see also \cite[Prop.\ 10.6.16]{Mo}).

\begin{thm} [Majid \cite{majid}]\label{thm:majid}
Let $H$ be a finite-dimensional Hopf algebra.
The categories of  Yetter-Drinfeld $H$-modules and of
$D(H)$-modules may be identified:
A  Yetter-Drinfeld $H$-module $M$ is a $D(H)$-module via
\begin{equation}\label{action-dual-general}
  (f\ot a)\cdot m = \sum \langle f \mid  (a\cdot m)_{(1)}\rangle (a\cdot m)_{(0)}
\end{equation}
for all $f\in H^*$, $a\in H$, $m\in M$.
Conversely, a $D(H)$-module $M$ is a Yetter-Drinfeld $H$-module via the restriction
of $M$ to a left $H$-module and to a left $H^*$-module (equivalently,  a right $H$-comodule).
\end{thm}

\noindent {\bf Radford's construction}\medskip

Radford   \cite{radford} gave a construction of all simple
Yetter-Drinfeld modules for certain graded Hopf algebras.
Although the results in  \cite{radford}  are more general, we will state them
only under the assumption that  $\k$ is an algebraically closed field of
characteristic 0.

\begin{lemma}[Radford~{\cite[Lem.\ 2]{radford}}]\label{lemma-beta-action}
Let $H$ be a bialgebra over an algebraically closed field  $\k$ of characteristic 0, and suppose $H^{\op}$ is a Hopf algebra
with antipode $S^{\op}$.
If $\beta \in \mbox{{\rm{Alg}}}_\k(H,\k)$, then $H_{\beta}= (H, \db, \Delta) \in  {}_H\mathcal {YD}^H$, where for all $x,a$ in $H$,

\begin{equation}\label{beta-action}
 x\db a =\sum \beta(x\_2)x\_3a S^{\op}(x\_1).
\end{equation}
If $g$ belongs to the grouplike elements $G(H)$ of $H$, then $H \db g$ is a Yetter-Drinfeld submodule of $H_{\beta}$.
\end{lemma}

Let  $H= \bigoplus_{n=0}^{\infty}H_{n}$ be a graded Hopf algebra with $H_0=\k G$ a group algebra, and $H_n=H_{n+1} = \cdots = (0)$ for some $n > 0$.   An algebra map $\beta : H \to \k$ is determined by its restriction to $G$, which is an element in the dual group $\widehat{G} = \Hom(G, \k^{\times})$ of all group homomorphisms from $G$ to the multiplicative group $\k ^{\times}$.
In particular, $\beta(H_n) = 0$ for all $n \geq 1$.

\begin{thm}[Radford {\cite[Cor.~1] {radford}}]  \label{radford}
Let $H= \bigoplus_{n=0}^{\infty}H_{n}$ be a graded Hopf algebra over an algebraically closed field $\k$ of characteristic 0.
Suppose that $H_0=\k G$ for some finite abelian group $G$ and $H_n=H_{n+1} = \cdots = (0)$ for some $n > 0$. Then
$$(\beta, g) \mapsto [H\db g]$$
is a bijective correspondence between the Cartesian product of sets
$\widehat{G} \times G $  and the set of isomorphism classes of simple Yetter-Drinfeld $H$-modules.
\end{thm}
Note that when $G$ is abelian, then for all $g,h\in G$,
$$h\db g =  \beta(h)hgh^{-1} = \beta(h)g.$$

\section{Cocycle twists}\label{sec:twists}

In this section, we first collect definitions and known results about cocycle twists.
Then we apply these results in the context of Radford's construction to obtain
in Theorem \ref{modules-twist} below an explicit correspondence of Yetter-Drinfeld modules
under twists.

\begin{defns}{\em \begin{itemize}\item [{\rm (a)}]
The {\em convolution product} of two $\k$-linear maps \newline  $\sigma, \sigma':
H\ot H\rightarrow \k$ is defined by
$$
  \sigma\sigma' (a\ot b) = \sum \sigma(a_{(1)}\ot b_{(1)}) \sigma'(a_{(2)}\ot b_{(2)})$$
for all $a,b\in H$.

\item [{\rm (b)}]   A $\k$-linear map
$\sigma: H\otimes H\to \k$ on a Hopf algebra $H$ is {\em convolution
invertible} if there exists a $\k$-linear map $\sigma^{-1}:H\otimes H\to \k$ such that
$$\qquad \quad
   \sum \sigma(a\_1\ot b\_1)\sigma^{-1}(a\_2\ot b\_2) = \varepsilon(ab) =  \sum \sigma^{-1}(a\_1\ot b\_1)\sigma(a\_2\ot b\_2) .
$$

\item [{\rm (c)}]  A {\em 2-cocycle} on $H$ is a convolution invertible
$\k$-linear map $\sigma \colon H \otimes H \to \k$ satisfying
$$\sum \sigma(a\_1 \otimes b\_1) \sigma(a\_2 b\_2 \otimes c) = \sum \sigma(b\_1 \otimes c\_1) \sigma(a \otimes b\_2c\_2)$$
and $\sigma(a\otimes 1) = \sigma (1 \otimes a) = \varepsilon(a)$, for all $a, \, b, \, c \, \in H$.
\item [{\rm (d)}]  Given a 2-cocycle $\sigma$ on $H$,  the {\em cocycle twist} $H^\sigma$  of $H$ by $\sigma$ has $H^\sigma = H$ as a coalgebra,  and  the algebra structure on $H^\sigma$  is given by

\begin{equation}\label{eq;algtwist}  a \cdot_{\sigma} b = \sum \sigma(a\_1 \otimes b\_1)a\_2 b\_2 \sigma^{-1}(a\_3 \otimes b\_3) \end{equation}
for all $a,b\in H^{\sigma}$.
The antipode of $H^{\sigma}$ is
\begin{equation}\label{eq;antitwist} S^{\sigma}(a) = \sum \sigma\left(a\_1 \otimes S(a\_2)\right) S(a\_3)\sigma^{-1}\left(S(a\_4) \otimes a\_5\right),
\end{equation}
where $S$ is the antipode of $H$.
\end{itemize}}
\end{defns}

Recall that similarly a {\em 2-cocycle} on a group $G$
is a map $\sigma \colon G \times G \to \k^{\times}$ such that for all $g,h,k\in G$,
\begin{equation}\label{eq;groupco}  \sigma(g,h) \sigma(gh, k) =  \sigma(h, k) \sigma(g , hk).\end{equation}
A 2-cocycle $\sigma$ on $G$
is a {\em 2-coboundary} if there exists a function $\gamma: G\rightarrow \k^{\times}$
such that $\sigma=d\gamma$ where for all $g,h\in G$,
\begin{equation*}
   d\gamma(g,h) :=  \gamma(g)\gamma(h)\gamma(gh)^{-1}.
\end{equation*}
The cocycle $\sigma$ is said to be {\em normalized} if $\sigma(g,1) = 1 = \sigma(1,g)$ for all $ g \in G$.

The set of all (normalized) 2-cocycles on $G$ forms an abelian group under
pointwise multiplication, and the 2-coboundaries form a subgroup.
The quotient group of 2-cocycles modulo 2-coboundaries is denoted
$\coh^2(G,\k^{\times})$ and is known as the {\em Schur multiplier} of $G$.

If $G$ is abelian,
$\beta: G\rightarrow \k^{\times}$ is a group homomorphism, and $\sigma$ is
a 2-cocycle on $G$, then the function
$\beta_{ g, \sigma} : G \to \k^{\times}$ given by
\begin{equation}\label{eq;betagh}\beta_{g, \sigma} (h) = \beta(h) \sigma(g,h) \sigma^{-1}(h,g)
\end{equation}
is also a group homomorphism for each $g \in G$; this may be verified directly
from the cocycle identity, and it follows as well from the general theory below.

Assume ${ H= \bigoplus_{n=0}^{\infty}H_n}$ is a graded Hopf algebra with $H_0= \k G$, where $G$ is a finite abelian group.
A normalized 2-cocycle $\sigma \colon  G \times G \to \k^{\times}$
extends linearly to a Hopf algebra 2-cocycle on the group algebra $\k G$ (we abuse notation and call this extension $\sigma$ also),
where  $\sigma(g\otimes h) = \sigma(g,h)$ for $g,\, h \in G$.
Let $\pi \colon H \to \k G$ be the projection onto $H_0$,  and let $\sigma_{\pi} \colon H \otimes H \to \k$ be defined by $\sigma_{\pi}= \sigma \circ (\pi \otimes \pi)$.
Then $\sigma_{\pi}$ is a Hopf algebra 2-cocycle on $H$
(see for example \cite{andr-schneider}). The convolution inverse of $\sigma_{\pi}$ is
$\sigma_{\pi}^{-1}=(\sigma^{-1})_{\pi}$ where $\sigma^{-1} \colon G
\times G \to \k^{\times}$ is given by $\sigma^{-1}(g,h) =  \sigma(g,h)^{-1}$ for
all $g,h\in G$.
By abuse of notation, we also denote $\sigma_{\pi}$ by $\sigma$
for simplicity.

The resulting cocycle twist  $H^{\sigma}$ is a graded Hopf algebra with the same grading as $H$,  and $\left(H^{\sigma}\right)_0 = \k G $ as a Hopf algebra.
Let  $\sigma '$ be another 2-cocycle on $G$ extended to $H$.  Since
$\sigma$ and  $\sigma '$ factor through the projection $\pi$
to $\k G$ (a cocommutative Hopf algebra), the convolution product $\sigma \sigma '$
   is again a 2-cocycle on $H$.   Suppose now that $d\gamma$
   is a 2-coboundary coming from the map $\gamma: G \rightarrow
   \k^\times$.  Then
$H^{(d\gamma)(\sigma)}\cong H^{\sigma}$ as Hopf algebras:
One may check that the map
$\psi: H^{(d\gamma)(\sigma)}\rightarrow H^{\sigma}$ defined by
$$
  \psi(a) = \sum \gamma(a_{(1)}) a_{(2)} \gamma(a_{(3)})^{-1}  \quad \hbox{\rm for all} \ a \in H,
$$
is an isomorphism with inverse given by
 $\psi^{-1}(a) = \sum \gamma(a_{(1)})^{-1} a_{(2)} \gamma(a_{(3)})$, where
 we assume $\gamma$ has been extended to $H$ by setting  $\gamma(H_n) = 0$ for all $n \neq 0$ and by letting $\gamma$ act linearly on $H_0 = \k G$
(compare  \cite{doi} or \cite[Thm.\ 7.3.4]{Mo}).

\begin{definition} {\em  If $M$ is  Yetter-Drinfeld $H$-module and
$\sigma$ is a 2-cocycle  of $H$, there is a corresponding Yetter-Drinfeld $H^{\sigma}$-module,
denoted $M^\sigma$, defined as follows:
It is $M$ as a comodule, and the $H^{\sigma}$-action is given by
$$a \cdot^{\sigma} m = \sum \sigma \left( (a\_2 \cdot m\_0)\_1 \otimes a\_1 \right)(a\_2 \cdot m\_0)\_0\sigma^{-1}(a\_3 \otimes m\_1),$$
for all $a\in H$, $m\in M$ (see \cite[(13)]{chen-zhang}).
}\end{definition}

We will use the following equivalence of categories of Yetter-Drinfeld modules.
In general, if $\sigma$ is a 2-cocycle on $H$, then $\sigma^{-1}$ is a 2-cocycle on
$H^{\sigma}$; this follows from \cite[(10)]{chen-zhang}, using the definition
of the multiplication on $H^{\sigma}$ in \eqref{eq;algtwist}.
(In the special case that $\sigma$ is induced from the group of grouplike elements
of $H$ as described above, which is the only case we consider in this paper,
$\sigma^{-1}$ is a 2-cocycle on $H$ as well.)
The following result is due originally to Majid and Oeckl \cite[Thm.~2.7]{majid-oeckl};
the formulation of it for left-right Yetter-Drinfeld modules can be found in 
Chen and Zhang \cite[Cor.~2.7]{chen-zhang}.

\begin{thm}
\label{CZ}
Let  $\sigma$ be a 2-cocycle on the Hopf algebra $H$. Then the categories ${}_{H}\mathcal{YD}^{^{H}}$
and ${}_{H^{\sigma}}\mathcal{YD}^{^{H^\sigma}}$ are monoidally equivalent under the
functor
\begin{equation}\label{functor}F_\sigma \colon {}_{H}\mathcal{YD}^{^{H}} \to {}_{H^{\sigma}}\mathcal{YD}^{^{H^\sigma}}\end{equation}
which is the identity on homomorphisms,  and on the objects is given by $F_\sigma(M)= M^\sigma$.
The inverse functor is given by $N \mapsto N^{\sigma^{-1}}$.
\end{thm}

In particular, we note that by the definition of the modules $M^{\sigma}$,
the category equivalence in the theorem preserves dimensions of modules.

We now apply Theorem \ref{radford} to obtain both
Yetter-Drinfeld $H$-modules and Yetter-Drinfeld $H^{\sigma}$-modules, under appropriate hypotheses on $H$.
The next theorem gives an explicit description of  the correspondence
of their simple Yetter-Drinfeld modules.

\begin{thm}\label{modules-twist}
Let $H= \bigoplus_{n=0}^{\infty}H_{n}$ be a graded Hopf algebra over an algebraically
closed field $\k$ of characteristic 0 for which
$H_0=\k G$,  $G$ is a finite abelian group, and $H_n=H_{n+1} = \cdots = (0)$ for some $n > 0$.
Let $\sigma \colon G \times G \to \k^{\times}$ be a normalized 2-cocycle
 on $G$, extended to a 2-cocycle on $H$.
For each pair $(\beta, g) \in \widehat{G} \times G$, we have
$$\left(H\db g\right)^{\sigma} \cong H^{\sigma} \DOT_{\beta_{g,\sigma}} g,$$
where $\beta_{g,\sigma}$ is defined in \eqref{eq;betagh}.  
The isomorphism maps $h\db^{\sigma} g$ to $h\DOT_{\beta_{g,\sigma}} g$ for each $h$
in $H$.
\end{thm}
\begin{proof}
Since $H \db g$ is a simple Yetter-Drinfeld $H$-module by Theorem \ref{radford},  and the inverse of
the functor $F_{\sigma}$ of (\ref{functor})
is the identity on homomorphisms,  $\left( H \db g\right)^{\sigma}$ is a simple Yetter-Drinfeld $H^{\sigma }$-module.

Now $H^{\sigma }$ is also a graded Hopf algebra with $
\left(H^{\sigma }\right)_0= \k G$ and $\left(H^{\sigma }\right)_n = \left(H^{\sigma }\right)_{n+1} = \cdots =0$.  Thus,  isomorphism classes of simple Yetter-Drinfeld $H^{\sigma }$-modules are also in one-to-one correspondence with $\widehat{G} \times G$, and we have unique $\beta' \in \widehat{G}$ and $g' \in G$ such that
$$\left( H \db g \right)^{\sigma } \cong H^{\sigma } \DOT_{\beta'} g'.$$
Let $\phi \colon \left( H \db g \right)^{\sigma } \to  H^{\sigma } \DOT_{\beta'} g'$ be an isomorphism of Yetter-Drinfeld modules. Since $\phi$ is a comodule map,
$$(\phi \otimes \id)( \Delta(g)) = \Delta(\phi(g)).$$
Applying $\varepsilon \otimes \id$ on both sides, we obtain
$$\varepsilon(\phi(g)) g = \phi(g).$$
Because $\phi(g)\neq 0$, this equation implies that $\varepsilon (\phi(g))\neq 0$, so $g \in \Im(\phi)$. However, there is a unique grouplike element in $H^{\sigma } \DOT_{\beta'} g'$; hence $g=g'$.
By multiplying $\phi$
by $\varepsilon(\phi(g))^{-1}$ if necessary, we may assume $\phi(g)=g$.
Note that $\beta'$ is uniquely determined by its images $\beta'(h)$ for all $h\in G$. Let $h \in G$, then
\begin{eqnarray*}
\beta'(h) g &=& h \DOT_{\beta'} g \; = \; h \DOT_{\beta'} \phi(g) \; = \; \phi \left( h \DOT_{\beta}^{\sigma} g \right) \\
&=& \phi\left( \sum \sigma\left( (h\db g)\_1\ot h \right) (h\db g)\_0 \sigma^{-1}(h\ot g)\right) \\
&=& \phi \left( \sum \sigma\left((\beta(h)g)\_1\ot h \right)(\beta(h)g)\_0\sigma^{-1}(h\ot g)\right) \\
&=& \phi\left(\sigma(g,h)\beta(h)g\sigma^{-1}(h,g)\right)\\
&=& \beta(h)\sigma(g,h)\sigma^{-1}(h,g)g.
\end{eqnarray*}
Therefore $\beta'(h) = \beta(h)\sigma(g,h)\sigma^{-1}(h,g)=
\beta_{g,\sigma}(h)$, as desired.
\end{proof}

\section{Restricted quantum groups}\label{sec:rqg}

We will apply Theorems \ref{radford}, \ref{CZ}, and \ref{modules-twist}
to some restricted quantum groups.
We focus on the two-parameter quantum groups $\ursn$ in this paper.
The same techniques may be used more generally on finite-dimensional
two-parameter or multi-parameter quantum groups (such as those for example  in  \cite{hu-wang,towber-westreich}). \bigskip

\noindent {\bf The quantum groups $\ursn$}\medskip

  Let $\e_1, \ldots, \e_n$ be an orthonormal basis of  Euclidean
space ${\mathbb R}^n$ with inner product $\langle \ , \ \rangle$.
Let $\Phi =
\{\e_i -\e_j \mid 1 \leq i \neq j \leq n\}$ and
$\Pi =
\{\alpha_j = \e_{j}- \e_{j+1} \mid j = 1, \dots, n-1\}$.
Then $\Phi$ is a finite
root system of type A$_{n-1}$ with base $\Pi$  of simple roots.
Let $r,s\in \k ^{\times}$ be roots of unity with $r\neq s$ and let
$\ell$ be the least common multiple of the orders of $r$ and $s$.
Let $q$ be a primitive $\ell$th root
of unity and $y$ and $z$ be nonnegative integers such that $r=q^y$ and $s=q^z$.
The following Hopf algebra, which appeared in
\cite{{benkart-witherspoon}},  is a slight modification of one defined by Takeuchi \cite{takeuchi}.
\begin{definition}\label{quantum-group}{\em The algebra $U=U_{r,s}(\mathfrak{sl}_{n})$ is the unital associative
$\k$-algebra
generated by $
\{e_j, \ f_j, \ \w_j^{\pm 1}, \ (\w_j')^{\pm 1}, \ \ 1 \leq j < n \}$, subject
to the following relations.
\medbreak

 \begin{itemize}
\item[(R1)]  The $\w_i^{\pm 1}, \ (\w_j')^{\pm 1}$ all commute with one
another and \\
 $\w_i \w_i^{-1}= \w_j'(\w_j')^{-1}=1,$
\item[(R2)] $ \w_i e_j = r^{\langle \epsilon_i,\alpha_j\rangle }s^{\langle
\epsilon_{i+1},\alpha_j\rangle}e_j \w_i$ \ \ and \ \ $\w_if_j =
r^{-\langle\epsilon_i,\alpha_j\rangle}s^{-\langle\epsilon_{i+1},\alpha_j\rangle} f_j\w_i,$
\item[(R3)]
$\w_i'e_j = r^{\langle\epsilon_{i+1},\alpha_j\rangle}s^{\langle\epsilon_{i},\alpha_j\rangle}e_j
\w_i'$ \ \ and \ \ $\w_i'f_j =
r^{-\langle\epsilon_{i+1},\alpha_j\rangle}s^{-\langle\epsilon_{i},\alpha_j\rangle}f_j \w_i'$,
\item[(R4)]
$\displaystyle{[e_i,f_j]=\frac{\delta_{i,j}}{r-s}(\w_i-\w_i'),}$
\item[(R5)] $[e_i,e_j]=[f_i,f_j]=0 \ \ \text{ if }\ \ |i-j|>1, $
\item[(R6)]
 $e_i^2e_{i+1}-(r+s)e_ie_{i+1}e_i+rse_{i+1}e_i^2 = 0,$ \\
$e_i e^2_{i+1} -(r+s)e_{i+1}e_ie_{i+1} +rse^2_{i+1}e_i = 0,$
\item[(R7)]
 $f_i^2f_{i+1}-(r^{-1}+s^{-1})f_if_{i+1}f_i +r^{-1}s^{-1}f_{i+1}f_i^2 =
 0,$\\
$f_i f^2_{i+1} -(r^{-1}+s^{-1})f_{i+1}f_if_{i+1}+r^{-1}s^{-1} f^2_{i+1} f_i=0, $
\end{itemize}
for all $1\leq i,j <n.$
}\end{definition}

The following coproduct, counit, and antipode give $U$ the structure of a Hopf algebra:
\begin{eqnarray*}
\Delta(e_i)=e_i\otimes 1+\omega_i\otimes e_i, &\ \ & \Delta(f_i)=1\otimes f_i +
f_i\otimes\omega_i',\\
\varepsilon(e_i)=0, &\ \ & \varepsilon(f_i)=0,\\
S(e_i)=-\omega_i^{-1}e_i,&\ \ & S(f_i)=-f_i(\omega_i')^{-1},
\end{eqnarray*}
and $\omega_i,\omega_i'$ are grouplike, for all $1\leq i <n$.

Let $U^0$ be the group algebra generated by all $\omega_i^{\pm 1}$,
$(\omega_i')^{\pm 1}$ and let $U^{+}$ (respectively, $U^{-}$) be the
subalgebra of $U$ generated by all $e_i$ (respectively, $f_i$).

Let $$\E_{j,j}=e_j \ \ \mbox{ and } \ \ \E_{i,j}=e_i\E_{i-1,j}-
  r^{-1}\E_{i-1,j}e_i \quad  (i>j),$$
$$\F_{j,j}=f_j \ \ \mbox{ and } \ \ \F_{i,j}=f_i\F_{i-1,j}-
s\F_{i-1,j}f_i \quad (i>j).$$
The algebra $U$ has a triangular decomposition $U\cong U^-
\otimes U^0\otimes U^+$ (as vector spaces), and, as shown in
\cite{b-k-l,kharchenko},  the subalgebras $U^+$, $U^-$
have monomial Poincar\'e-Birkhoff-Witt (PBW) bases given respectively by
 \begin{equation*}
\mathcal E: = \{\E_{i_1,j_1}\E_{i_2,j_2}\cdots\E_{i_p,j_p}\mid (i_1,j_1)\leq
(i_2,j_2)\leq\cdots\leq (i_p,j_p)\mbox{ lexicographically}\},
\end{equation*}
\begin{equation*}
\mathcal F:= \{\F_{i_1,j_1}\F_{i_2,j_2}\cdots\F_{i_p,j_p}\mid (i_1,j_1)\leq
(i_2,j_2)\leq\cdots\leq (i_p,j_p)\mbox{ lexicographically}\}.
\end{equation*}

In \cite{benkart-witherspoon} it is proven that all
$\E^{\ell}_{i,j}$, $\F^{\ell}_{i,j}$, $\omega_i^{\ell}-1$, and
$(\omega'_i)^{\ell}-1$ ($1\leq j \leq i< n$) are central
in $U_{r,s}({\mathfrak{sl}}_{n})$.   The ideal $I_{n}$
generated by these elements is a Hopf ideal
\cite[Thm.\ 2.17]{benkart-witherspoon}, and so the quotient
\begin{equation}\label{resqg}{\mathfrak{u}}_{r,s}({\mathfrak{sl}}_{n})=
  U_{r,s}({\mathfrak{sl}}_{n})/I_{n}\end{equation}
is a Hopf algebra, called the {\em restricted two-parameter quantum group}.
Examination of the PBW bases $\mathcal E$ and $\mathcal F$ shows that ${\mathfrak{u}}_{r,s}({\mathfrak{sl}}_{n})$ is finite dimensional,  and in fact,  it
is a pointed Hopf algebra \cite[Prop.\ 3.2]{benkart-witherspoon}.
 \medbreak

Let $\mathcal E_\ell$ and $\F_{\ell}$ denote the sets of monomials in $\mathcal E$ and $\F$ respectively, in which each $\mathcal E_{i,j}$ or $\F_{i,j}$ appears as a factor at most $\ell-1$ times. Identifying cosets in $\mathfrak{u}_{r,s}({\mathfrak{sl}}_{n})$ with their
representatives, we may assume $\mathcal E_{\ell}$ and $\F_{\ell}$ are bases for the subalgebras of
${\mathfrak{u}}_{r,s}({\mathfrak{sl}}_{n})$ generated by the elements $e_i$ and $f_i$ respectively.

\begin{definition} {\em Let $\mathfrak b_{r,s}$ denote the Hopf subalgebra of ${\mathfrak{u}}_{r,s}(
{\mathfrak{sl}}_n)$ generated by all $\omega_i, e_i$,
and let ${\mathfrak{b}}'_{r,s}$ denote the Hopf subalgebra generated by all $\omega_i',\, f_i$ ($1\leq i <n$). }
\end{definition}

Under certain conditions on the parameters $r$ and $s$,  stated explicitly in the next
theorem,
$\ursn$ is  a Drinfeld double.
The statement of the result differs somewhat from  that given in
\cite{benkart-witherspoon} due to the use of a slightly
different definition of the Drinfeld double;   details on these differences can be found in
\cite{pereira2}.

\begin{thm}
[{\cite[Thm.~4.8] {benkart-witherspoon}}] \label{double}
Assume $r = q^y$ and $s = q^z$,
where $q$ is a primitive $\ell$th root of unity, and
\begin{equation}\label{gcd-condition} \gcd(y^{n-1}-y^{n-2}z+\cdots +(-1)^{n-1}z^{n-1},\ell)=1.
\end{equation}
Then there is an
isomorphism of Hopf algebras ${\mathfrak{u}}_{r,s}({\mathfrak{sl}}_{n})
\cong (D((\mathfrak{b}_{r,s}')^{\coop}))^{\coop}$.
\end{thm}

For simplicity we set \begin{equation}\label{eq:hrsG} H_{r,s}:=(\mathfrak{b}_{r,s}')^{\coop} \ \ \hbox{\rm and} \ \
G = G(H_{r,s}) = \langle \w_i' \mid 1 \leq i < n\rangle.\end{equation}
Then $H_{r,s}$ is a graded Hopf algebra with $\w_i' \in (H_{r,s})_0$, and $f_i \in (H_{r,s})_1$ for all $1 \leq i <n$.
When the conditions of Theorem \ref{double} are satisfied,
we may apply Theorems \ref{thm:majid} and \ref{radford} to obtain
$\ursn$-modules as Yetter-Drinfeld $H_{r,s}$-modules via Radford's construction as follows.
Note that algebra maps from $H_{r,s}$ to $\k$ are the grouplike elements of $H_{r,s}^*$.

\begin{thm}
Assume (\ref{gcd-condition}) holds.
Isomorphism classes of simple $\ursn$-modules (equivalently, simple Yetter-Drinfeld $H_{r,s}$-modules) are in one-to-one correspondence with $G(H_{r,s}^*) \times G(H_{r,s})$.
\end{thm}

\noindent {\bf 2-cocycles on Borel subalgebras} 
\medskip

The group $G=G(H_{r,s})$
is isomorphic to $(\Z/\ell \Z)^{n-1}$.
We wish to determine all cocycle twists of $H_{r,s}$
arising from cocycles of $G$ as described in Section \ref{sec:twists}.
By Theorems \ref{thm:majid} and \ref{CZ}, the categories of modules of the Drinfeld doubles of all such cocycle
twists are equivalent via an equivalence that preserves comodule
structures.

The cohomology group ${\rm H}^2(G,\k^{\times})$ is known
from the following result of Schur \cite{schur}, which
can also be found in \cite[Prop.\ 4.1.3]{karpilovsky}.

\begin{thm}[Schur \cite{schur}] \ ${\rm H}^2((\Z/\ell\Z)^{m}, \k^{\times}) \cong (\Z/\ell\Z)^{\binom{m}{2}}$.
\end{thm}

The isomorphism in the theorem is given by induction, using
\cite[Thm.\ 2.3.13]{karpilovsky}:
\begin{eqnarray*}
  {\rm H}^2((\Z/\ell\Z)^i,\k^{\times}) &\cong &
        {\rm H}^2((\Z/\ell\Z)^{i-1},\k^{\times})\times {\rm Hom}
   ((\Z/\ell \Z)\ot_{\Z} (\Z/\ell\Z)^{i-1},\k^{\times})  \\
 &\cong &
        {\rm H}^2((\Z/\ell\Z)^{i-1},\k^{\times})\times {\rm Hom}
   ((\Z/\ell\Z)^{i-1},\k^{\times}),
\end{eqnarray*}
for $i\geq 2$,
which uses the fact that $(\Z/\ell \Z)\ot_{\Z} (\Z/\ell\Z)^{i-1} \cong (\Z/\ell\Z)^{i-1} $.
This isomorphism is given as follows:
Identify $(\Z/\ell\Z)^i$ with $(\Z/\ell\Z) \times (\Z/\ell\Z)^{i-1}$.
Let $\psi: (\Z/\ell\Z)^{i-1}\times
 (\Z/\ell\Z)^{i-1} \rightarrow \k^{\times}$ be a 2-cocycle,
and $\phi:  (\Z/\ell\Z)^{i-1} \rightarrow \k^{\times}$ be
a group homomorphism.
The corresponding 2-cocycle $\sigma:  (\Z/\ell\Z)^{i}\times
 (\Z/\ell\Z)^{i} \rightarrow \k^{\times}$ is
$$
   \sigma((g,h), (g',h')) = \psi(h,h') \phi(g\cdot h')
$$
for all $g,g'\in \Z/\ell\Z$ and $h,h'\in (\Z/\ell\Z)^{i-1}$,
where $g\cdot h'$ is the image of $g\ot h'$ under the isomorphism
 $(\Z/\ell \Z)\ot_{\Z} (\Z/\ell\Z)^{i-1} \cong (\Z/\ell\Z)^{i-1} $, i.e.\
it is the action of $g\in \Z/\ell\Z$ on the element $h'$ in the
$(\Z/\ell\Z)$-module $(\Z/\ell\Z)^{i-1}$.
The asymmetry in the formula is due to the fact that $\coh^2(\Z/\ell\Z, \k^{\times})=0$.

\begin{example}{\em
If $i=2$, and the generators for  $(\Z/\ell\Z)^2$ are $g_1$ and $g_2$,
we obtain
$$
  \sigma(g_1^{i_1}g_2^{i_2}, g_1^{j_1}g_2^{j_2}) = \psi(g_2^{i_2},g_2^{j_2})\phi(g_1^{i_1}\cdot g_2^{j_2}) = \psi(g_2^{i_2},g_2^{j_2}) \phi(g_2^{i_1j_2}).
$$
Since all 2-cocycles on the cyclic group $\Z/\ell\Z$ are coboundaries, we may
assume without loss of generality that $\psi$ is trivial, so that $\sigma$ is given
by the group homomorphism $\phi: \Z/\ell\Z\rightarrow \k^{\times}$.
If $q$ is a primitive $\ell$th root of unity, any such group homomorphism
just sends the generator of $\Z/\ell\Z$ to a power of $q$, and thus
we obtain representative cocycles, one for each $a\in \{0,1,\ldots,\ell -1\}$:
$$
  \sigma(g_1^{i_1}g_2^{i_2},g_1^{j_1}g_2^{j_2})= q^{a i_1 j_2}.
$$
}
\end{example}

Similarly, by induction, we have the following.

\begin{prop} \label{upper-triangular}
\  Let  $q$  be a primitive $\ell$th root of unity.  \ A set of cocycles $\sigma: (\Z/\ell\Z)^{m}\rightarrow \k^{\times}$
which represents all elements of $\coh^2((\Z/\ell\Z)^m,\k^{\times})$ is given by
\begin{equation}\label{eq:cocycle-formula}
  \sigma((g_1^{i_1}\cdots g_m^{i_m}, \ g_1^{j_1}\cdots g_m^{j_m}))=
  q^{\left(\sum_{1 \leq k < l \leq m}   a_{k,l} \, i_k j_l\right)}
\end{equation}
where $0\leq a_{k,l}\leq \ell -1$.
Thus the cocycles are parametrized by
$m \times m$ strictly upper triangular matrices with entries in $\Z/\ell\Z$.
\end{prop}

Next  we determine  isomorphisms among cocycle twists of the Hopf
algebras $H_{r,s}$. We will denote by $|r|$ the order of $r$ as a
root of unity, and use $\hbox{\rm lcm}(|r|, |s|)$ to denote the least common
multiple of the orders of $r$ and $s$.
Fix $n\geq 2$, and let $G= \langle \w_i' \mid 1 \leq i < n \rangle$, the grouplike
elements of $H_{r,s} = (\b'_{r,s})^{\coop} \subset \ursn^{\coop}$.

\begin{thm}\label{prop:findim}  The following are equivalent: 
\begin{itemize}
\item[{\rm (i)}]
There is a cocycle $\sigma$ induced from $G$ such that $H_{r,s}^{\sigma} \cong H_{r',s'}$
as graded Hopf algebras.
\item[{\rm (ii)}] $\mbox{\rm lcm}(|r|,|s|)=\mbox{\rm lcm}(|r'|,|s'|)$ and 
 $r' (s')^{-1}=rs^{-1}$.
\end{itemize}
\end{thm}

We note that if $\mbox{lcm}(|r|,|s|)\neq\mbox{lcm}(|r'|,|s'|)$,
then the dimensions of $H_{r,s}$ and of $H_{r',s'}$ are different.

\begin{proof}
First assume that $\mbox{\rm lcm}(|r|,|s|) =\mbox{\rm lcm}(|r'|,|s'|)$ and
$r'(s')^{-1}=rs^{-1}$. Let $\sigma=\sigma_{r,r'}$ be the cocycle defined by
\begin{equation}\label{eq:sigrr} \sigma(\w_1'^{i_1} \cdots \w_{n-1}'^{i_{n-1}}, \w_1'^{j_1} \cdots \w_{n-1}'^{j_{n-1}})= (r(r')^{-1})^{\sum_{k=1}^{n-2} i_kj_{k+1}}. \end{equation}
To compare this with the expression in  (\ref{eq:cocycle-formula}),
write $r(r')^{-1} = q^a$ for some $a$, and note that this cocycle coincides with the
choice $a_{k,l} = a \delta_{l, k+1}$.
We claim that $H_{r,s}^{\sigma}\cong H_{r',s'}$.
We will prove this by verifying relations
(R3) and (R7) of Definition \ref{quantum-group}, as well
as the nilpotency of certain PBW basis elements.
Each of the other relations is either automatic or does not apply to $H_{r,s}$.
First we check that the Serre relations (R7)  for $H_{r',s'}$ hold in $H_{r,s}^{\sigma}$:
\begin{eqnarray*}
&&  \hspace{-1cm}
  f_i\cdot_{\sigma} f_i\cdot_{\sigma} f_{i+1} - \big((r')^{-1}+(s')^{-1}\big)f_i\cdot_{\sigma} f_{i+1}\cdot_{\sigma} f_i
  + (r')^{-1}(s')^{-1} f_{i+1}\cdot_{\sigma} f_i\cdot_{\sigma} f_i \\
&&\hspace{-.6cm}= \sigma(\omega_i',\omega_i')\sigma((\omega_i')^2,\omega_{i+1}) f_i^2f_{i+1} - ((r')^{-1}+(s')^{-1})\sigma(\omega_i',\omega_{i+1}')\sigma(\omega_i'\omega_{i+1}',
\omega_i')f_if_{i+1}f_i\\
  &&\hspace{2cm} + (r')^{-1}(s')^{-1}\sigma(\omega_{i+1}',\omega_i')\sigma(\omega_i'\omega_{i+1}',\omega_i') f_{i+1}f_i^2\\
&&\hspace{-.6cm}= \ (r(r')^{-1})^2 f_i^2 f_{i+1} - ((r')^{-1}+(s')^{-1})(r(r')^{-1}) f_if_{i+1}f_i + (r')^{-1}(s')^{-1}f_{i+1}f_i^2\\
&&\hspace{-.6cm}= \ (r(r')^{-1})^2\big(f_i^2 f_{i+1} - (r^{-1}+s^{-1}) f_if_{i+1}f_i + r^{-1}s^{-1} f_{i+1}f_i^2\big)\\
&&\hspace{-.6cm}= \ 0,
\end{eqnarray*}
by (R7) for $H_{r,s}$. Similarly, the other Serre relation in (R7) holds.

Next let $$\F_{i,j}^{\sigma} = f_i \cdot_\sigma \F_{i-1,j}^{\sigma} - s' \F_{i-1,j}^{\sigma} \cdot_{\sigma}f_i$$ for $i >j$ and set $\F_{i,i}^{\sigma} = f_i.$   We need to show that $(\F_{i,j}^{\sigma})^{\ell}=0$ in $H_{r,s}^{\sigma}$ (and that no smaller power of $\F_{i,j}^{\sigma}$ is zero). In what follows we show by induction on $i-j$ that $\F_{i,j}^{\sigma}= \F_{i,j}$.
Let $\w_{i,j}'= \w'_j \cdots \w'_i.$

\begin{itemize}
\item If $i=j$, then $\F_{i,i}^{\sigma}= f_{i}= \F_{i,i}$.
\item Assume $\F_{i-1,j}^{\sigma} = \F_{i-1,j}$, then
\begin{eqnarray*}
\F_{i,j}^{\sigma} &=& f_i \cdot_\sigma \F_{i-1,j}^{\sigma} - s' \F_{i-1,j}^{\sigma} \cdot_{\sigma}f_i \\
&=&  f_i \cdot_\sigma \F_{i-1,j} - s' \F_{i-1,j} \cdot_{\sigma}f_i \\
&=& \sigma(\w_i', \w'_{i-1,j})f_i\F_{i-1,j} - s'\sigma(\w_{i-1,j}', \w_i')\F_{i-1,j}f_i \\
&=& f_i\F_{i-1,j} - s'r(r')^{-1}\F_{i-1,j}f_i \\
&=&  f_i\F_{i-1,j} - s\F_{i-1,j}f_i  \ \ = \ \ \F_{i,j}.
\end{eqnarray*}
\end{itemize}
We will use the following identity which is a consequence of \cite[Lem.~2.22] {benkart-witherspoon}:
$$(\pi \otimes \id \otimes \pi) \Delta^{2}(\F_{i,j})= \w'_{i,j} \otimes \F_{i,j} \otimes 1,$$
where $\pi$ is the projection onto $\k G$ and  $G = \langle \omega_i' \mid 1\leq i <  n \rangle$.
Then $$\F_{i,j} \cdot_{\sigma} \F_{i,j} = \sigma(\w'_{i,j}, \w'_{i,j})\F_{i,j}^{2}$$ and
$\underbrace{\F_{i,j} \cdot_{\sigma} \cdots \cdot_{\sigma} \F_{i,j}}_{m} =
\left(\sigma(\w'_{i,j}, \w'_{i,j})\right)^{\binom{m}{2}} \F_{i,j}^m = 0 $
if and only if $\F_{i,j}^m=0.$

Finally  we check that the second relation in (R3) of Definition
 \ref{quantum-group} for $H_{r',s'}$ holds in $H_{r,s}^{\sigma}.$
Note that the second relation in (R3) for $H_{r,s}$ can be written as

$$\w'_if_j=b_{i,j}\,f_j\w'_i, \
\ \hbox{\rm where} \  \
 b_{i,j}=
\begin{cases}
s, & j= i-1 \\
rs^{-1}, & j=i \\
r^{-1}, & j=i+1 \\
1, & \text{otherwise.}
\end{cases}$$

\noindent Hence, we need to verify that $\w'_i \cdot_{\sigma}f_j=c_{i,j}\, f_j\cdot_{\sigma}\w'_i,$
where
$
c_{i,j}=
\begin{cases}
s', & j= i-1 \\
r'(s')^{-1}, & j=i \\
(r')^{-1}, & j=i+1 \\
1, & \text{otherwise;}
\end{cases}$

\begin{eqnarray*}
\w'_i \cdot_{\sigma}f_j =\sigma(\w'_i, \w'_j)\w'_if_j &=&
\begin{cases}
r(r')^{-1} \w'_if_j, & j= i+1 \\
 \w'_if_j, & \text{otherwise}
\end{cases} \\
&=&
\begin{cases}
b_{i,j}r(r')^{-1} f_j\w'_i, &  j= i+1 \\
b_{i,j} f_j\w'_i, & \text{otherwise}
\end{cases} = \begin{cases}
s \, f_j\w'_i, & j= i-1 \\
rs^{-1}f_j\w'_i & j=i \\
(r')^{-1}f_j\w'_i & j= i+1 \\
f_j\w'_i & \text{otherwise.}
\end{cases}
\end{eqnarray*}

On the other hand
\begin{eqnarray*}
c_{i,j} \, f_j \cdot_{\sigma} w'_i &=& c_{i,j} \sigma(\w'_j,\w'_i)\, f_j\w'_i =
\begin{cases}
r(r')^{-1}c_{i,j} f_j\w'_i, & j = i-1 \\
c_{i,j} f_j\w'_i, & \text{otherwise}
\end{cases}
\\
&=&
\begin{cases}
s \,f_j\w'_i, &j= i-1 \\
rs^{-1} f_j\w'_i, &j=i \\
(r')^{-1}f_j\w'_i, & j= i +1 \\
f_j\w'_i& \text{otherwise.}
\end{cases}
\end{eqnarray*}
Thus $H_{r,s}^{\sigma}\cong H_{r',s'}$.

Now assume that $H_{r,s}^{\sigma}\cong H_{r',s'}$ as graded Hopf algebras
for some cocycle $\sigma$ induced from $G$.
Then the dimensions of these two algebras are the same, and so 
$\mbox{\rm lcm}(|r|,|s|) = \mbox{\rm lcm}(|r'|,|s'|)$.
Since cohomologous cocycles give isomorphic twists,
we may assume without loss of generality that $\sigma$ is given by formula (\ref{eq:cocycle-formula}).
Since $\deg(f_1)=1$,  the element $f_1$ gets mapped
under the isomorphism $H_{r,s}^{\sigma}\stackrel{\sim}{\rightarrow} H_{r',s'}$  to another skew-primitive element of degree 1, which must be
a scalar multiple of some $f_i$.
This forces $\omega_1'$ to be mapped to $\omega_i'$.
Then the second relation in (R3) of Definition \ref{quantum-group} for $H_{r,s}^{\sigma}$ is
$\omega_1'\cdot_{\sigma} f_1 = \sigma(\omega_1',\omega_1') \omega_1' f_1 = rs^{-1}f_1\omega_1'
=rs^{-1}f_1\cdot_{\sigma}\omega_1'$, while the second relation in (R3) for $H_{r',s'}$ is $\omega_i'f_i = r'(s')^{-1}f_i\omega_i'$.
The existence of the isomorphism now implies  that $r'(s')^{-1}=rs^{-1}$.  \end{proof}

\begin{remark}{\rm The cocycles $\sigma$ in \eqref{eq:sigrr}  do not exhaust all possible
cocycles, as can be seen by comparing  \eqref{eq:sigrr}  with Proposition \ref{upper-triangular}.
Other cocycles
will take $H_{r,s}$ to multi-parameter versions of the Borel subalgebra not considered here.
For comparison, examples of some different  finite-dimensional multi-parameter quantum groups
of $\mathfrak{sl}_n$-type with different Borel subalgebras appear in \cite{towber-westreich}.} \end{remark}

\begin{remark}\label{rem:YDcorr}{\rm The correspondence between simple Yetter-Drinfeld modules
resulting from Theorem
\ref{prop:findim}  can be realized explicitly as follows.  Assume that $\mbox{\rm lcm}(|r|,|s|)=\mbox{\rm lcm}(|r'|,|s'|)$ and $r'(s')^{-1} = rs^{-1}$.   Let $y,y'$ be such that
$r = q^y$, $r' = q^{y'}$.    Recall that by Radford's result (Theorem \ref{radford}),
the simple Yetter-Drinfeld modules for
$H = H_{r,s}$ have the form $H \db g$ where $g \in G$, and $\beta \in \widehat G$.
Let $g = (\w_1')^{d_1} \cdots (\w_{n-1}')^{d_{n-1}}$ and suppose $\beta \in \widehat
G$ is defined by integers $\beta_i$ for which $\beta(\w_i') = q^{\beta_i}$.   Then by Theorem \ref{modules-twist},
the correspondence is given by  $\left(H_{r,s}\db g \right)^\sigma = H_{r',s'} \DOT_{\beta_{g,\sigma}} g$,  where
\begin{equation}\label{eq:bgsig} \beta_{g,\sigma}( \w_i') = q^{\beta_{i}} (r(r')^{-1})^{d_{i-1}}(r(r')^{-1})^{-d_{i+1}}= q^{\beta_i + (d_{i-1}-d_{i+1})(y-y')}\end{equation}
and $d_0=d_n=0$.
} \end{remark}

\begin{cor} Let $q$ be a primitive $\ell$th root of unity. Assume $r,s$ are powers
of $q$. 
\begin{itemize}
\item[\rm (a)]  If $|rs^{-1}| = \mbox{\rm lcm}(|r|,|s|)$, then  
there is a one-to-one correspondence
between simple Yetter-Drinfeld modules for $H_{r,s}$ and those for $H_{1,r^{-1}s}$.
\item[\rm (b)]  If $\ell$ is a prime and $r \neq s$,
 there is a one-to-one correspondence
between simple Yetter-Drinfeld modules for $H_{r,s}$ and those for $H_{1,r^{-1}s}$.
Hence,  the simple Yetter-Drinfeld modules of  $H_{r,s}$  and $H_{r',s'}$
(and also the simple modules for $D(H_{r,s})$ and $D(H_{r',s'})$)
are in bijection for any two pairs $(r,s)$, $(r',s')$, with $r \neq s$ and $r' \neq s'$.
In particular, when (\ref{gcd-condition}) holds for both pairs, there is a one-to-one
correspondence between simple modules for $\ursn$ and those for ${\mathfrak{u}}_{r',s'}({\mathfrak{sl}}_{n})$.
\end{itemize}
\end{cor}

\begin{proof}  The proof of Theorem \ref{prop:findim} shows that when
$\mbox{\rm lcm}(|r|,|s|)=\mbox{\rm lcm}(|r'|,|s'|)$ and $r'(s')^{-1} = rs^{-1}$, then
$H_{r,s}^\sigma = H_{\xi r, \xi s} = H_{r',s'}$, where $\xi = r' r^{-1}$ and $\sigma = \sigma_{r,r'}$.
Applying this to the case $r' = 1$, $s' = r^{-1}s$,  $\xi = r^{-1}$, we get
$H_{r,s}^\sigma = H_{1,r^{-1}s}$.  The correspondence between modules then comes from
Theorem \ref{modules-twist} as in Remark~\ref{rem:YDcorr}.
The first statement in part (b) follows
immediately from (a) since $|r^{-1}s| = \mbox{\rm lcm}(|r|,|s|) = \ell$.
Now observe that the representation theory of $H_{1,q}$ for any primitive $\ell$th
root $q$ of 1 is the same as that of any other.
\end{proof}

\medskip

If  $rs^{-1}$ is a primitive
$\ell$th root of unity for $\ell$ odd, we may assume $rs^{-1}=q^2$, where $q$ is a primitive $\ell$th root of unity.   Then if $\mbox{\rm lcm}(|r|,|s|)= \ell$,   the algebra $H_{q,q^{-1}}$ is a cocycle
twist of $H_{r,s}$. Combined with Theorem \ref{double}, this
implies the following:

\begin{thm}\label{sequence}
Let $\ell$ be odd and
suppose $\mbox{\rm lcm}(|r|,|s|)= \ell = |rs^{-1}|$, and $q$ is a primitive
$\ell$th root of 1 with $rs^{-1}=q^2$.  Assume that hypothesis (\ref{gcd-condition}) holds,
where  $r = q^y$ and $s = q^z$. Then
the categories of $\ {\mathfrak{u}}_{r,s}({\mathfrak{sl}}_{n})$-modules, ${\mathfrak{u}}_{q,q^{-1}}({\mathfrak{sl}}_{n})$-modules,
and  $\ {\mathfrak{u}}_{1,r^{-1}s}({\mathfrak{sl}}_{n})$-modules are all
monoidally equivalent.
\end{thm}

\begin{rem}\label{rem:factor}{\rm  The one-parameter quantum group $\mathfrak{u}_{q}(\mathfrak{sl}_n)$ is 
the quotient of
${\mathfrak{u}}_{q,q^{-1}}({\mathfrak{sl}}_n)$ by the ideal generated by all $\omega_i'-\omega_i^{-1}$.  Under the hypotheses of the theorem, each simple ${\mathfrak{u}}_{q,q^{-1}}({\mathfrak{sl}}_n)$-module factors
uniquely as a tensor product of a one-dimensional ${\mathfrak{u}}_{q,q^{-1}}({\mathfrak{sl}}_n)$-module and a simple module
for $\mathfrak{u}_{q}(\mathfrak{sl}_n)$, as long as  $\mbox{gcd}(n, \ell) =1$
by \cite[Thm.~2.12]{pereira2}.
The converse holds in general:
Any $\mathfrak{u}_{q}(\mathfrak{sl}_n)$-module becomes a
${\mathfrak{u}}_{q,q^{-1}}({\mathfrak{sl}}_n)$-module via the quotient map
${\mathfrak{u}}_{q,q^{-1}}({\mathfrak{sl}}_n) \rightarrow \mathfrak{u}_{q}(\mathfrak{sl}_n)$,
and we may take the tensor product of a simple $\mathfrak{u}_{q}(\mathfrak{sl}_n)$-module with any one-dimensional
${\mathfrak{u}}_{q,q^{-1}}({\mathfrak{sl}}_n)$-module to get a simple
${\mathfrak{u}}_{q,q^{-1}}({\mathfrak{sl}}_n)$-module.
The one-dimensional modules are described in the next section.}
\end{rem}

\section{Computations}

Throughout this section we assume $r = q^y, s = q^z$, where $q$ is a primitive $\ell$th root of unity
and  $\mbox{\rm lcm}(|r|,|s|) = \ell$.  We begin by computing all the
one-dimensional Yetter-Drinfeld modules for $H_{r,s}$.   When  (\ref{gcd-condition}) is
satisfied, this gives the one-dimensional modules for $\mathfrak{u}_{r,s}(\mathfrak{sl}_n)$.
Then we specialize to the case $n = 3$.    In  \cite{pereira1,pereira-web},  the computer algebra system {\sc Singular::Plural} \cite{plural} was used to construct simple $\mathfrak{u}_{r,s}({\mathfrak{sl}}_3)$-modules
for many values of $r$ and $s$.  Here we give a brief discussion of
the calculations from the point of view of the results of this paper and present some of the  examples.   The examples illustrate
how widely different the dimensions of the simple $\mathfrak{u}_{r,s}({\mathfrak{sl}}_3)$-modules
can be when the Borel subalgebras are not related by cocycle twists.
\medskip

The one-dimensional Yetter-Drinfeld modules for $H= H_{r,s}$  have the
form $H \db g$ for some $g \in G = \langle \omega_i' \mid 1 \leq i < n\rangle$.
Combining the definition
of the $\db$ action in (\ref{beta-action}) with  the coproduct formulas in $H$, we have
for all $x\in H$ and $g\in G(H)$,
\begin{equation}\label{action_f1}
f_i \db x =  xS(f_i) + \beta(\w_i')f_ix(\w'_i)^{-1} = -xf_i(\w_i')^{-1} + \beta(\w_i')f_ix(\w_i')^{-1}, 
\end{equation}
and 
 \begin{equation}\label{eq:action_w1}
\w'_i \db g = \beta(\w_i')\w_i'g(\w'_i)^{-1} = \beta(\w_i')g.
\end{equation}
Since $H \db g$ is one-dimensional, $H \db g = \k g$.   Now if $g = (\omega_1')^{d_1}  \cdots (\omega_{n-1}')^{d_{n-1}}$,   then
\begin{equation*} g f_i = r^{d_i-d_{i-1}} s^{d_{i+1}-d_i} f_i g, \end{equation*}
where $d_0 = 0 = d_n$.   Thus, equation (\ref{action_f1}) for $x=g$ implies that

\begin{equation}\label{eq:fgswap}  f_i \db g = -r^{d_i-d_{i-1}} s^{d_{i+1}-d_i} f_i g
(\omega_i')^{-1} + \beta(\omega_i')f_i g (\omega_i')^{-1} \qquad \hbox{\rm for} \ \ 1 \leq i < n.
\end{equation}\

\noindent Since by the PBW basis theorem, the right side is a multiple of $g$
only if it is zero, we obtain that
\begin{equation}\label{eq:betadef}  \beta(\omega_i') = r^{d_i-d_{i-1}} s^{d_{i+1}-d_i} \qquad \hbox{\rm for} \ \  1 \leq i < n.
\end{equation}
Hence, each module of the
form $H \db g$ which is one-dimensional  is determined by $g$.
Conversely,  every $g = (\omega_1')^{d_1}  \cdots (\omega_{n-1}')^{d_{n-1}}  \in G$ determines a one-dimensional module $H \db g$ where
 $\beta$ is given by (\ref{eq:betadef}).    Therefore,
the number of nonisomorphic one-dimensional modules is
$\ell^{n-1}$.   In summary we have the following result, due originally
to the second author (compare \cite[Prop. 2.10]{pereira2}).

\begin{prop}  Assume $r=q^y$, $s=q^z$, where $q$ is a primitive $\ell$th root of unity
and $\mbox{\rm lcm}(|r|,|s|) = \ell$.
Then the one-dimensional Yetter-Drinfeld modules for $H_{r,s}:=(\mathfrak{b}_{r,s}')^{\coop}$ are
in bijection with the elements of $G = \langle \omega_i' \mid 1 \leq i < n\rangle$, under the correspondence  taking $g =(\omega_1')^{d_1}  \cdots (\omega_{n-1}')^{d_{n-1}}$ to the one-dimensional module
$H \db g = \k g$, where $\beta$ is given by (\ref{eq:betadef}).
Thus, when (\ref{gcd-condition}) holds, the one-dimensional modules for
$\ursn$ are in bijection with the elements of $G$.   \end{prop}
\medskip

\noindent {\bf Techniques used in {\sc  Singular::Plural}}\medskip

We now discuss computer calculations done for small values of
$\ell$ and $n = 3$.   In order to use {\sc Singular::Plural}, the
algebras must be given by generators and relations of
a particular form which allows computations using Gr\"obner bases.
Details on the types of algebras involved may be found in \cite{Apel,levan}.

 Let $\mathcal{B}'$ be the subalgebra of $U_{r,s}(\mathfrak{sl}_3)$ generated by $\left\{f_1, f_2, \w_1', \w_2'\right\}$.
Adding the element $\F_{21}=f_2f_1-sf_1f_2$ to the generating set, and rewriting the relations of Definition \ref{quantum-group}, we see that
$\mathcal{B}'$ is generated by

$\left\{x_1=f_1,\, x_2=\F_{21},\, x_3=f_2,\, x_4=\w_1', \, x_5=\w_2'\right\}$, subject to relations
$$\left\{ x_jx_i= C_{ij}x_ix_j +D_{ij}\mid  1 \leq i <j \leq 5 \right\}$$
where the coefficients $C_{ij}$ and elements $D_{ij}$ are given as follows:
\begin{enumerate}
\item [(1)]  $C_{12}=C_{23}=C_{25}=r$,
\item [(2)] $C_{13} = C_{15} =s$,
\item [(3)] $C_{24}=s^{-1}$,
\item [(4)]  $C_{34}=r^{-1}$,
\item [(5)] $C_{14}= C_{35} =rs^{-1}$,
\item [(6)] $C_{45} = 1$,
\item [(7)] $D_{ij}=0$ if $(i,j)\neq (1,3)$ and $D_{13}=\F_{21}$.
\end{enumerate}

Let $\mathcal{I}$ be the two-sided ideal of $\mathcal{B}'$ generated by the set $$\left\{ (\w_1')^{\ell} -1, \, (\w_2')^{\ell}-1,\, f_1^{\ell},\, \F_{21}^{\ell}, \, f^{\ell}_2 \right\}, $$
so that $H=H_{r,s}=(\mathfrak{b}_{r,s}')^{\coop}= \mathcal{B}'/\mathcal{I}$.

 Equation \eqref {eq:action_w1} shows that if $g \in G(H)$, then $H\db g$ is spanned by $$\left\{ (f_1^{k}\F_{21}^tf_2^m) \db g \mid 0 \leq k,\, t, \, m < \ell \right\}.$$

Applying equation (\ref{action_f1}) recursively, we may define a
procedure {\verb Beta } so that if $0\leq k,\, t, \, m < \ell$, \, $h \in H$, and $\beta: H \to \k$ is an algebra map given by $\beta(f_1)=q^a$ and $\beta(f_2)=q^b$, then {\verb Beta(a,b,k,t,m,h) } gives $(f_1^k\F_{21}^tf_2^m) \db h$.
Fix a grouplike element $g= (\w'_1)^{d_1}(\w'_2)^{d_2} \in H$.
In what follows we will construct a basis and compute the dimension of the module $H\db g$.
Some of the code for computing these bases, written by the second author, may be found in \cite{pereira1,pereira-web}.
 Let $$\F_{\ell} =\{f_1^{k}\F_{21}^{t}f_2^{m} \mid  \, 0 \leq k, \, t , \, m < \ell\}$$
 (so that $H\db g = \mathsf{span}_\k \{ f \db g \mid \, f \in \F_{\ell} \}$).
Consider the linear map $T_{\beta}:  \mathsf{span}_\k \F_{\ell} \to H$ given by $T_{\beta}(f) = f \db g$, and construct the matrix $M$ representing $T_{\beta}$ in the bases $\F_{\ell}$ and $\left\{ fh \mid  f \in \F_{\ell}, h\in G(H)\right\}$ of $\mathsf{span}_\k\F_{\ell}$ and $H$ respectively. Then $\dim(H \db g) = \mbox{rank}(M)$, and the nonzero columns of the column-reduced Gauss normal form of $M$ give the coefficients for the elements of a basis of $H\db g$.

Note that
$\dim(H)=\ell^5$ and $\dim\left(\mathsf{span}_\k \F_{\ell}\right)={\ell}^3$, so the size of $M$ is ${\ell}^5 \times {\ell}^3$. Computing the Gauss normal form of these matrices is an expensive calculation even for small values of $\ell$ such as $\ell=5$. However, by some reordering of $\F_{\ell}$ and of the PBW basis of $H$, $M$ is block diagonal. We explain this next.

For a monomial
$m=f_1^{a_1}\F_{21}^{a_2}f_2^{a_3}(\w'_1)^{a_4}(\w_2')^{a_5}$, let $\mbox{deg}_1(m) = a_1 + a_2$ and $\mbox{deg}_2(m) = a_2 + a_3$.
Since $\F_{21} = f_2f_1 - sf_1f_2$ and (\ref{action_f1}) is homogeneous in $f_1$
and in $f_2$,
$m \db x$ is a linear combination of monomials whose degrees are  $\mbox{deg}_i(m) + \mbox{deg}_i(x)$.
 For all $0 \leq u, v < 2\ell$, let $$A_{(u,v)}= \{ f \in \F_{\ell} \mid \mbox{deg}_1(f)=u \mbox{ and } \mbox{deg}_2(f)=v \}$$  and $$B_{(u,v)}= \{ f(\w_1')^{-u}(\w'_2)^{-v}g \mid f \in A_{(u,v)} \}. $$ Then for all $f \in A_{(u,v)}$, we have $f \db g \in \mathsf{span}_\k B_{(u,v)}$.
 The possible pairs $(u,v)$ are such that $0 \leq u, v \leq 2(\ell -1)$,  and if $u >v$ (resp. $v>u$), then $|u-v|$ is the minimum power of $f_1$ (resp. $f_2$) that must be a factor of a monomial in $A_{(u,v)}$.  
Therefore $|u-v| \leq \ell -1$;  that is, $u -( \ell -1) \leq v \leq u + \ell - 1$.
Another way of describing the sets $A_{(u,v)}$ and $B_{(u,v)}$ is as follows.
\begin{eqnarray*} A_{(u,v)} &=& \{ f_1^{u-i}\F_{21}^if_2^{v-i} \mid  i \in \N \mbox{ and } 0\leq u-i,\, i, \,v-i \leq  \ell -1 \} \\
&=& \{ f_1^{u-i}\F_{21}^if_2^{v-i} \mid  i \in \N \mbox{ and } \mathsf{n}_{u,v} \leq i \leq \mathsf{m}_{u,v} \},
\end{eqnarray*}
where $\mathsf{n}_{u,v}=\max(0, \ell-1-u, \ell -1 -v)$ and $\mathsf{m}_{u,v}= \min(\ell-1, u, v)$.
Since $(\w_i')^{-1}=(\w'_i)^{\ell-1}$, if $g=(\w_1')^{d_1}(\w'_2)^{d_2}$ we also have
$$B_{(u,v)}= \{  f(\w_1')^{(\ell-1)u+d_1}(\w'_2)^{(\ell-1)v+d_2} \mid  f \in A_{(u,v)} \}.$$
It is clear that $\F_{\ell} = \bigcup_{(u,v)} A_{(u,v)}$, a disjoint union,
and $H\db g= \bigoplus_{(u,v)} \mathsf{span}_\k B_{(u,v)}$. Therefore the disjoint union of bases for $\left(\mathsf{span}_\k A_{(u,v)}\right) \db g$ for all possible pairs $(u,v)$ gives a basis for $H\db g$, and thus $\dim (H \db g) = \sum_{(u,v)} \dim\left ((\mathsf{span}_\k A_{(u,v)}) \db g\right)$.
\bigskip

\noindent {\bf Examples of computations}
\medskip

We now present just a few examples computed using these techniques; more examples
may be found in \cite{pereira-web}.
These examples show that for some values of the parameters, we obtain a significantly different representation theory
for the two-parameter quantum group $\urs$ than for any one-parameter quantum group $\ufn$.

\begin{example}\label{example1}{\em
Let $q$ be a primitive 4th root of 1 and take $n=3$.
We compare the dimensions of the simple Yetter-Drinfeld modules
for $H_{1,q}$ and $H_{q,q^{-1}}$.
These Hopf algebras have the same dimension, but
by Theorem \ref{prop:findim}, there is no 2-cocycle $\sigma$ for which $(H_{1,q})^{\sigma}$ is isomorphic to $H_{q,q^{-1}}$.
Considering the dimension sets below, we see that the representation theories of the two algebras are quite different.  The results are displayed as multisets of dimensions, where each dimension is raised to the
number of nonisomorphic simple Yetter-Drinfeld modules of that dimension.  
\begin{eqnarray}\label{eq: dims} & \hskip - 2 truein \dim \left( H_{1,q} \db g \right),  \, g \in G(H_{1,q}), \, \beta \in \widehat{G(H_{1,q})}:  \\
&\qquad \qquad \qquad \{ 1^{16}, 3^{32}, 6^{32}, 8^{16}, 10^{32},12^{32}, 24^{32}, 26^{16}, 42^{32}, 64^{16}\} \nonumber \end{eqnarray}
 \begin{equation}\label{eq: dims2}  \dim \left( H_{q,q^{-1}} \db g \right), \,  g \in G(H_{q,q^{-1}}), \, \beta \in \widehat{G(H_{q,q^{-1}})}:  \quad \ \ 
  \{1^{16}, 3^{32}, 8^{16}, 16^{96}, 32^{96} \}. \end{equation}

\medskip

\noindent If for some $\gamma \in \widehat{G(H)}$ and $g\in G(H)$, \ $H\DOT_\gamma g$ is one-dimensional,  then $(H \DOT_{\gamma^{-1}}g^{-1})
\otimes (H \DOT_\gamma  g)
\cong H \DOT_\varepsilon 1 \cong H/\ker \varepsilon$, where $H \DOT_{\gamma^{-1}}g^{-1}$ is the one-dimensional module constructed from $g^{-1} \in G(H)$ (compare \cite[p.~961] {pereira2}).
Since 
$\left(H\DOT_\varepsilon 1\right) \otimes M \cong M$
for all modules $M \in$\YD ,  tensoring a simple Yetter-Drinfeld module with a one-dimensional module
yields another simple Yetter-Drinfeld module of the same dimension.    Because there are 16 one-dimensional modules for each of these algebras, this accounts for the fact that the superscripts 
in \eqref{eq: dims}  and \eqref{eq: dims2} are multiples
of 16.  

 The remaining cases $H_{r,s}$, where $r,s$ are 4th roots of 1 and $\mbox{lcm}(|r|,|s|)=4$,
are all related to the ones in \eqref{eq: dims} and \eqref{eq: dims2}:
By Theorem \ref{prop:findim}, $H_{q,q^2}$, $H_{q^2,q^3}$, and $H_{q^3,1}$ are
all cocycle twists of $H_{1,q}$.   Now $q^3$ is also a primitive 4th root of 1, and so the representation theory of
$H_{1,q^3}$ is the same  as that of $H_{1,q}$, and cocycle twists of $H_{1,q^3}$
are $H_{q,1}$, $H_{q^2,q}$, and $H_{q^3,q^2}$.
Finally, $H_{q^3,q}$ is a cocycle twist of $H_{q,q^3}$.
In fact, it follows from the quantum group isomorphisms in \cite[Thm.\ 5.5]{hu-wang}, that 
$H_{1,q} \cong H_{q^3,1}$, 
 $H_{q,1} \cong H_{1,q^3}$, $H_{q^2,q} \cong H_{q^3,q^2}$,  and $H_{q,q^2}\cong H_{q^2,q^3}$.}
\end{example}

\begin{example}[Dimensions of simple $\mathfrak{u}_{r,s}(\mathfrak{sl}_3)$-modules]\label{ex:dims}
{\em
Assume $n=3$, $\mbox{gcd}(6,\ell)=1$, $\mbox{lcm}(|r|,|s|)=\ell$, $rs^{-1}$ is a
primitive $\ell$th root of 1,  and $q$ is a primitive $\ell$th root of 1 such that
$q^2=rs^{-1}$. By Theorem \ref{prop:findim},
 $H_{r,s}^{\sigma}\cong H_{q,q^{-1}}$.
The dimensions of the simple Yetter-Drinfeld $H_{q,q^{-1}}$-modules are
given by the dimensions of simple $\mathfrak{u}_{q}(\mathfrak{sl}_3)$-modules;  the latter can be found in \cite{dobrev}.
Hence, this gives formulas for the dimensions of the simple
$\mathfrak{u}_{r,s}(\mathfrak{sl}_3)$-modules when it is a Drinfeld double.
If it is not a Drinfeld double, we just get the dimensions for the simple
Yetter-Drinfeld $H_{r,s}$-modules and for the simple modules of its double $D(H_{r,s})$.

Write $r=q^y$ and $s=q^z$. Let $g= \w_1'^{d_1}\w_2'^{d_2} \in G$ and $\beta: \k G \to \k$ be an algebra map given by $\beta(\w_i') = q^{\beta_i}$.
Taking $\sigma= \sigma_{r,q}$ (see the proof of Theorem \ref{prop:findim}),
we have  $(H_{r,s} \db g)^{\sigma} \cong H_{q,q^{-1}} \DOT_{\gamma} g$ by Theorems \ref{modules-twist} and \ref{prop:findim} and \eqref{eq:bgsig}, where $\gamma(\w_i')=q^{\gamma_i}$ with $$\gamma_1= \beta_1 -d_2(y-1) \quad \mbox{and} \quad \gamma_2=\beta_2+d_1(y-1).$$
Since $\mbox{gcd}(6, \ell)=1$, we have $ H_{q,q^{-1}} \DOT_{\gamma} g \cong M \otimes N$ with $M$ a simple $\mathfrak{u}_{q}(\mathfrak{sl}_3)$-module and $N$ a one-dimensional
$\mathfrak{u}_{q, q^{-1}}(\mathfrak{sl}_3)$-module by \cite[Thm.~2.12]{pereira2} (compare Remark \ref{rem:factor}). In particular,  $\mbox{dim}\left(H_{r,s} \db g\right) = \mbox{dim}\left(H_{q,q^{-1}} \DOT_{\gamma} g \right) = \mbox{dim}(M)$.
By the proof of \cite[Thm.~ 2.12]{pereira2},
$M = H_{q,q^{-1}} \DOT_{\beta_h} h$, where $h=\w_1'^{c_1}\w_2'^{c_2}$, $(c_1,c_2)$ a
solution of the following system of equations in $\Z/\ell \Z$,
\begin{equation*}
\left(
\begin{array}{lr}
\ \ \Id &\ \ \ \mathsf{A}^{-1} \\
-\mathsf{A} &\Id \hfil
\end{array}
\right)
\left(
\begin{array}{c}
c_1\\c_2\\ \chi_1\\ \chi_2
\end{array}
\right)
=
\left(
\begin{array}{c}
d_1\\ d_2 \\ \gamma_1\\ \gamma_2
\end{array}
\right),
\end{equation*}
and $\mathsf{A}$ is the Cartan matrix
\begin{equation*}
\mathsf{A}=\left(
\begin{array}{lr}
\ \ 2&-1\\
-1& \ 2
\end{array}
\right);
\end{equation*}
$\beta_h$ is defined by $\beta_h(\omega_1')=q^{c_2-2c_1}$,
$\beta_h(\omega_2')=q^{c_1-2c_2}$;  and the one-dimensional
module $N$ is given by the character $\chi$ of $G(H_{q,q^{-1}})$
specified by $\chi(\omega_1')=q^{\chi_1}$,
$\chi(\omega_2')=q^{\chi_2}$.

Let $m_1$ and $m_2$ be defined by
$$ m_1 \equiv (2c_1 -c_2 +1) \text{mod }\ell,  \quad  m_2 \equiv (2c_2 -c_1+1)\text{mod }\ell \quad \mbox{and} \quad 0 < m_i \le \ell.$$
By \cite{dobrev} (see the summary at \cite{pereira-web}) and the above
discussion, we have:
\begin{itemize}
\item
If $m_1 + m_2 \leq \ell$, then $$\dim(H_{r,s}\db g) = \half \Big(m_1m_2(m_1+m_2)\Big).$$
\item
If $m_1 + m_2 > \ell$ and  $m_i'=\ell - m_i$, then
$$\dim (H_{r,s}\db g) = \half \Big(m_1m_2(m_1+m_2)\Big)- \half \Big(m_1'm_2'(m_1'+m_2')\Big) .$$\end{itemize}

The parameters $m_1$ and $m_2$ can be obtained from the original data:
\begin{equation}\label{eq: dims3}m_1\equiv (d_1-d_2 +\half(d_2y-\beta_1) + 1)\,  \text{mod }\ell,  \ \  m_2 \equiv (d_2 -\half(yd_1 + \beta_2) + 1)\, \text{mod }\ell.
\end{equation}
}
\end{example}

\begin{remark}
{\rm Suppose in addition to the assumptions of Example \ref{ex:dims} that hypothesis (\ref{gcd-condition}) also holds.  Let $G_C$ denote the set of central grouplike elements of $\mathfrak{u}_{r,s}(\mathfrak{sl}_3)$,  and
let $\mathcal J$ be the ideal of $\mathfrak{u}_{r,s}(\mathfrak{sl}_3)$ generated by  $\{h-1 \mid h \in G_C\}$. It was shown in \cite[Prop.~2.7]{pereira2} that the simple modules for the quotient $\overline{\mathfrak{u}_{r,s}(\mathfrak{sl}_3) }=\mathfrak{u}_{r,s}(\mathfrak{sl}_3) /\mathcal J$ are of the form $H \DOT_{\beta_g} g$,  where if $g= \w_1'^{d_1}\w_2'^{d_2}$, then $\beta_g(\w_i')=q^{\beta_i}$ with $\beta_1= (z-y)d_1 + d_2y$
and $\beta_2=-zd_1+(z-y)d_2$. Using the formulas in (\ref{eq: dims3}) and the fact that $y-z \equiv 2\,  \text{ mod } \ell$ (since $q^2=rs^{-1}$),  we get
$$m_1\equiv (2d_1-d_2+1)\, \text{mod } \ell ,  \qquad m_2\equiv(2d_2-d_1+1)\, \text{mod } \ell.$$
Thus, for $\overline{\mathfrak{u}_{r,s}(\mathfrak{sl}_3)}$-modules,  the dimensions obtained in Example \ref{ex:dims} are those conjectured in \cite[Conjecture III.4]{pereira1}. Furthermore, under these hypotheses, every simple $\mathfrak{u}_{r,s}(\mathfrak{sl}_3)$-module is a tensor product of a one-dimensional module with a simple $\overline{\mathfrak{u}_{r,s}(\mathfrak{sl}_3) }$-module
by \cite[Thm.\ 2.12]{pereira2}. Thus, in order to determine all possible dimensions of simple $\mathfrak{u}_{r,s}(\mathfrak{sl}_3)$-modules, it is enough to use the formulas in Example \ref{ex:dims} with $m_1\equiv (2d_1-d_2+1) \text{mod } \ell$ and $m_2\equiv(2d_2-d_1+1)\, \text{mod } \ell$, for all $d_1, d_2 \in \left\{0, 1, \cdots, \ell-1\right\}$.
}
\end{remark}
\begin{center} {\bf Acknowledgments}  \end{center}
We thank Naihong Hu for providing us with a copy of \cite{hu-wang}
and Yorck Sommerh\"auser for bringing the paper \cite{majid-oeckl} to our attention.   
The second author is grateful for support from the grant
ANII FCE 2007-059.
The third author is grateful for support from the following grants:
NSF DMS-0443476, NSA H98230-07-1-0038, and NSF DMS-0800832.

\end{document}